\documentclass[12pt]{article}
\usepackage{amsmath,amsfonts,amssymb,amsthm}

\theoremstyle{definition}

\newtheorem{example}{Example}[section]
\newtheorem{remark}{Remark}[section]

\theoremstyle{plain}
\newtheorem{theorem}{Theorem}[section]

\providecommand{\prob}{\mathsf P}

\numberwithin{equation}{section}

\newcommand{\ve}{\varepsilon}

\newcommand{\R}{\mathbb R}

\newcommand{\E}{\textsf E}

\newcommand{\Sub}{\mathrm{Sub}}

\begin{document}
\begin{center}
{\Large\bf Investigation of sample paths properties of sub-Gaussian type random fields, with application to stochasic heat equations}
 \end{center}

\begin{center}
{{\bf Olha Hopkalo\\
{\it\small Department of Economic Cybernetics, Faculty of Economics, \\Taras Shevchenko National University of Kyiv, Kyiv, Ukraine, \\E-mail:~olha.hopkalo@knu.ua} \\
 Lyudmyla Sakhno\\
{\it\small
 Department of Probability Theory, Statistics and Actuarial Mathematics, \\Faculty of Mechanics and Mathematics, \\Taras Shevchenko National University of Kyiv, Kyiv, Ukraine, \\E-mail:~lyudmyla.sakhno@knu.ua}    
  }}
\end{center}

\begin{abstract} 
\noindent The paper presents bounds for the distributions of suprema for a particular class of sub-Gaussian type random fields defined over spaces with anisotropic metrics. The results are applied to random fields related to stochastic heat equations with fractional noise:  bounds for the tail distributions of suprema and estimates for the rate of growth are provided for such fields. 


\end{abstract}

\section{Introduction and preliminaries}

\bigskip

\noindent{\bf Aim and motivation.} 
In this paper, we study sample paths properties of a class of sub-Gaussian type random fields $X(t)$, $t\in T$, focusing on the case where a parameter set $T$ is endowed with an anisotropic metric and imposing some kind of H\"{o}lder continuity condition on the field $X$. Our aim is to establish upper bounds for the distribution of the supremum $P \big\{ \sup_{t\in T} |X(t)| > u \big\}$ for bounded $T$ and to evaluate a rate of growth of $X$ over an unbounded domain $V$ by considering upper bounds for $P \big\{ \sup_{t\in V} \frac{|X(t)|}{f(t)} > u \big\}$ for a properly chosen continuous function $f$. The study is motivated by applications to random fields related to stochastic heat equations. Extensive recent investigations of such equations resulted, in particular,  in establishing the H\"{o}lder continuity of solutions in various settings. It is quite natural and appealing to make further steps towards considerarion of different functionals of solutions. We evaluate the distribution of suprema of solutions and their asymptotic rate of growth.

\medskip
\noindent{\bf Approach and tools.} We present bounds for distributions of suprema assuming $X$ to belong to a particular class of $\varphi$-sub-Gaussian random fields (to be defined below), which provides a generalization of Gaussian and sub-Gaussian fields. To derive the results we apply entropy methods. Recall that the entropy approach in studing sample paths of a stochastic process $X(t)$, $t\in T$, requires to evaluate entropy characteristics of the set $T$ with respect to a particular metrics generated by the process $X$. The origins of this approach are due to Dudley, who stated coditions for boundedness  of
Gaussian processes in the form of convergence of metric entropy integrals (which we call now Dudley entropy integrals). We address for  corresponding references, e.g., to \cite{Adler} and \cite{BK}, where in the latter one the entropy approach was extended to different classes of processes, more general than Gaussian ones.

\medskip
\noindent{\bf Some facts from the general theory of $\varphi$-sub-Gaussian random variables and fields.}
Note that it is important for applications to go beyond the Gaussianity assumption in considered models and possible extensions are provided by sub-Gaussian and $\varphi$-sub-Gaussian random processes and fields.  Recall that a random variable $\xi$ is sub-Gaussian if its moment generating function is majorized by that of a Gaussian centered random variable $\eta\sim N(0,\sigma^2)$: $$\E\exp(\lambda\xi)\le \E\exp(\lambda\eta)=\exp(\sigma^2\lambda^2/2).$$ 
The generalization of this notion to the classes of $\varphi$-sub-Gaussian random variables is introduced as follows (see, \cite[Ch.2]{BK}, \cite{GKN}, \cite{KO}, \cite{Vasylyk}).

Consider a continuous even convex
function $\varphi$ such that
$\varphi(0)=0$, $\varphi(x)>0$ as $x\neq0$ and 
$\lim\limits_{
        x\to 0}\frac{\varphi(x)}{x}=0$, $\lim\limits_{
x\rightarrow \infty 
}\frac{\varphi(x)}{x}=\infty.$ Note that such functions are called  Orlicz N-functions. Suppose that  $\varphi$  satisfies 
  additionally:
    $\lim \inf\limits_{
            x\to 0 
    }\frac{\varphi(x)}{x^2}=c>0,$ where
 the case $c=\infty$ is possible.
 
Let $\varphi$ be the function with the above properties and $\{\Omega, \cal{F}, \prob \}$ be a standard probability space. The random variable $\zeta$ is  $\varphi$-sub-Gaussian, or belongs to the space $\Sub_{\varphi} (\Omega)$,   if $\E \zeta=0,$ $ \E \exp\{\lambda\zeta\}$ exists for all $\lambda\in\R$ and there exists a constant  $a>0$ such that the following inequality holds for all $\lambda\in\R$
\begin{equation}\label{def1}\E \exp\{\lambda \zeta\} \leq \exp \{\varphi(\lambda a)\}.\end{equation}
The  random field $X(t)$, $t\in T$, is called  $\varphi$-sub-Gaussian  if the  random variables
$\{X(t), t\in T\}$ are  $\varphi$-sub-Gaussian.

The space $ \Sub_\varphi(\Omega)$ is a Banach space with respect to
the norm (see \cite{GKN,KO}):
$$ \tau_\varphi (\zeta) = \inf\{a > 0: \E \exp \{\lambda \zeta \} \leq
\exp\{ \varphi (a \lambda)\}.$$

 For $\varphi$-sub-Gaussian random variable $\zeta$  the following estimate for its tail probability holds:
\begin{equation}\label{tail}
P\{|\zeta|>u\}\le 2\exp\left\{-\varphi^*\left(\frac{u}{\tau_\varphi(\zeta)}\right)\right\}, \,\,u>0,
\end{equation}
where the function $\varphi^*$
defined by
 $\varphi^*(x)=\sup_{\substack{
       y\in \R}}(xy-\varphi(y))$
 is called  the Young-Fenchel transform (or Legendre transform, or convex conjugate) of the function $\varphi$.  
It was stated in \cite{BK} (Corollary 4.1, p. 68) that, moreover, a random variable
$\zeta$ is a $\varphi$-sub-Gaussian if and only if  $\E\zeta=0$ and there exist constants $C>0$, $D>0$ such that 
\begin{equation}\label{tail1}
P\{|\zeta|>u\}\le C\exp\left\{-\varphi^*\left(\frac{u}{D}\right)\right\}.
\end{equation}

This second characterization of $\varphi$-sub-Gaussian random variable by the tail behavior of its distribution is important for practical applications.

The class of $\varphi$-sub-Gaussian random variables includes  centered compactly supported distributions, reflected Weibull distributions, centered bounded distributions, Gaussian, Poisson distributions. In the case when $\varphi=\frac{x^2}{2}$, the notion of $\varphi$-sub-Gaussianity reduces to the classical sub-Gaussianity.
The main theory for the spaces of $\varphi$-sub-Gaussian random variables and stochastic processes was presented in \cite{BK, GKN, KO} followed by numerous further studies. 
 Various  classes of $\varphi$-sub-Gaussian processes and fields were studied, in particular, in \cite{BKOS, HS, KO2016,  KOSV2018, KOSV, Sakhno2022}.

The property of $\varphi$-sub-Gaussianity allows to evaluate different functionals of the stochastic processes, in particular, the behavior of their suprema.

Estimates for distribution of supremum $P \{\sup_{t\in T} |X(t)| > u\}$ of $\varphi$-sub-Gaussian stochastic process $X$ were derived in various forms in the monograph \cite{BK} basing on entropy methods.

We will base our study on the following theorem (see \cite{BK}, Theorems 4.1--4.2, pp. 100, 105).

\begin{theorem} [\cite{BK}]\label{sup_X}
Let $X(t)$, $t \in T,$ be a $\varphi$-sub-Gaussian process and
$\rho_X$ be the pseudometrics  generated by $X$, that is,
$\rho_X (t, s) = \tau_\varphi (X(t) - X(s)),$ $ t, s \in T.$  Suppose further that 

\begin{enumerate}
\item[(i)] \vskip-2truemm the pseudometric space $(T, \rho_X)$ is separable, the process $X$ is separable on $(T, \rho_X)$;

\item[(ii)] \vskip-2truemm $\varepsilon_0 :=\sup\limits_{t\in T} \tau_\varphi (X(t)) < \infty$;

\item[(iii)] \vskip-5truemm
\begin{equation}\label{condI}
 I_\varphi(\varepsilon_0):=\int\limits_0^{\varepsilon_0}
\Psi (\ln(N(v)))\,d v < \infty,
\end{equation}
where $
\Psi(v)=\frac{v}{\varphi^{(-1)}(v)}$ and $N(v)=N_{\rho_X}(v)$, $v>0$, is the metric massiveness of the pseudometric space $(T, \rho_X)$, that is,   the smallest number of  elements in a
$v$-covering of $T$ by closed balls, w.r.t. the metric $\rho_X$, of a radius at most $v$.
\end{enumerate}

\vskip-2truemm 
Then for all $\lambda>0$ and $0 < \theta <1$
\begin{equation}
\E \exp \{\lambda\sup_{{t}\in T}|X(t)| \}
\leq   2Q(\lambda, \theta),
\end{equation}
where
\begin{equation}
Q(\lambda, \theta)=\exp\Big\{ \varphi\Big( \frac{\lambda\varepsilon_0}{1-\theta}\Big) + \frac{2\lambda}{\theta(1-\theta)} I_\varphi (\theta \varepsilon_0)\Big\},
\end{equation}
and for $0 < \theta <1$, $u>\frac{2I_\varphi (\theta \varepsilon_0)}{\theta(1-\theta)}$
\begin{equation} P \big\{ \sup_{ {t} \in T} |X( {t} )| > u \big\}
\leq   2A(u, \theta), 
\end{equation}
 where
\begin{equation} A(u, \theta) = \exp\Big\{ -\varphi^* \Big( \frac{1}{\ve_0} \Big( u(1-\theta) -
\frac{2}{\theta} I_\varphi (\theta \varepsilon_0) \Big) \Big) \Big\}.\end{equation}
\end{theorem}

In the above theorem and in what follows we denote by $f^{(-1)}$  the inverse function for a function $f$.

The integrals of the form \eqref{condI} are called entropy integrals. Entropy characteristics of the parameter set $T$ with respect to the pseudometric $\rho_X$ generated by the process $X$ and the rate of growth of the metric massiveness $N(v)=N_{\rho_X}(v)$, $v>0$, or metric entropy $H(v):=\ln(N(v)$ are important for the study of sample paths properties of the underlying process $X$ (see \cite{BK}).

Consider now a metric space $(T, d)$, with  an arbitrary metric $d$ and suppose that this metric space is separable. 
Suppose that we can evaluate the metric massivness $N_d$ of $T$ with respect to the metric $d$ and also have a bound for the function   
$\rho_X (t, s) = \tau_\varphi (X(t) - X(s))$  in terms of $d(t,s)$. Then Theorem \ref{sup_X} implies the following result.

\begin{theorem}\label{metric_d}
Let $X(t)$, $t \in T,$ be a $\varphi$-sub-Gaussian process and
$T$ be supplied with a metric $d$.  Assume that 
\begin{enumerate}
\item[(i)] \vskip-2truemm the metric space $(T, d)$ is separable, the process $X$ is separable on $(T, d)$;

\item[(ii)] \vskip-2truemm $\varepsilon_0 :=\sup\limits_{t\in T} \tau_\varphi (X(t)) < \infty$;

\item[(iii)] \vskip-2truemm
there exists a monotonically increasing continuous function $\sigma(h)$, $0 < h \leq \sup_{t,s\in T} d(s,t)$ such
 that $\sigma(h)\rightarrow 0$ as $ h \rightarrow 0 $ and
 \vskip-2truemm
\begin{equation}\label{suptau}
\sup_{\substack{
        d({t},{s})\leq h,\\
        {t}, {s}\in T \\
}}\tau_{\varphi} (X ({t})-
X ({s}))\leq \sigma(h),
\end{equation}
and for $0<\varepsilon\le\gamma_0$
\begin{equation}\label{condI2}
 \widetilde I_\varphi(\varepsilon):=\int\limits_0^{\varepsilon}
\Psi (\ln (N_d(\sigma^{(-1)}(v)))\,d v < \infty,
\end{equation}
where $
\Psi(v)=\frac{v}{\varphi^{(-1)}(v)}$, $N_d(v)$, $v>0$, is the metric massiveness of the metric space $(T, d)$, $\gamma_0=\sigma(\sup_{t,s\in T}d(t,s))$.
\end{enumerate}

\vskip-2truemm
Then the statement of Theorem \ref{sup_X} holds for $\lambda>0$ and $0 < \theta <1$ such that $\theta\varepsilon_0\le\gamma_0$ with $Q(\lambda, \theta)$ and $A(u, \theta)$ changed for the $\widetilde Q(\lambda, \theta)$ and $\widetilde A(u, \theta)$ which correspond to the integral $\widetilde I_\varphi$:
\begin{equation}
\E \exp \{\lambda\sup_{{t}\in T}|X(t)| \}
\leq   2\widetilde Q(\lambda, \theta)
\end{equation}
and for $u>\frac{2\widetilde I_\varphi (\theta \varepsilon_0)}{\theta(1-\theta)}$
\begin{equation} P \big\{ \sup_{ {t} \in T} |X( {t} )| > u \big\}
\leq   2 \widetilde A(u, \theta), 
\end{equation}
 where
 \begin{equation}
\widetilde Q(\lambda, \theta)=\exp\Big\{ \varphi\Big( \frac{\lambda\varepsilon_0}{1-\theta}\Big) + \frac{2\lambda}{\theta(1-\theta)} \widetilde I_\varphi (\theta \varepsilon_0)\Big\},
\end{equation}
\begin{equation} \widetilde A(u, \theta) = \exp\Big\{ -\varphi^* \Big( \frac{1}{\ve_0} \Big( u(1-\theta) -
\frac{2}{\theta} \widetilde I_\varphi (\theta \varepsilon_0) \Big) \Big) \Big\}.\end{equation}
\end{theorem}

\begin{proof}Theorem \ref{metric_d} follows immediately from Theorem \ref{sup_X}. We have from \eqref{suptau} that \\ $\sup_{\substack{
        d({t},{s})\leq h,\\
        {t}, {s}\in T \\
}}\rho_X(t,
s)\leq \sigma(h)$, therefore, the smallest number of elements in an $\varepsilon$-covering of $(T, \rho_X)$ can be bounded by the smallest number of elements in a $\sigma^{(-1)}(\varepsilon)$-covering of $(T, d)$: $N_{\rho_X}(\varepsilon)\le N_{d}(\sigma^{(-1)}(\varepsilon))$. This implies the estimate $I_\varphi(\varepsilon)\le \widetilde I_\varphi(\varepsilon)$, as $\varepsilon\le\gamma_0$, and the statement of the theorem follows.
\end{proof}

Theorem \ref{metric_d} has been mainly used in the literature with a choice of the metric space $(T, d)$ of the form: $T = \{ a_i \leq t_i \leq b_i,\, i=1, 2 \}$
and $ d({t}, {s}) = \max\limits_{i=1,2} |t_i -s_i|,$ ${t} =(t_1, t_2),$
${s} =(s_1, s_2).$ (see, for example, \cite{BKOS, KOSV2018, HSV} for application to the analysis of solutions to heat equation and higher order heat-type equations with random initial conditions).

In the  paper we study a particular class of $\varphi$-sub-Gaussian random fields with $\varphi=\frac{|x|^\alpha}{\alpha}$, $\alpha\in(1,2]$, which as a natural generalization of Gaussian and sub-Gaussian random fileds.  Gaussian and sub-Gaussian cases are involved in our consideration,  with the choice $\alpha=2$.  We study the sample paths of such fields for the case of the parameter set $T$ of the form $T=[a_1, b_1]\times[a_2, b_2]$ or $T=[0,+\infty)\times[-A, A]$ with  the so-called anisotropic metric $d(t,s)=\sum_{i=1,2}|t_i-s_i|^{H_i}$, $H_i\in(0,1]$, $i=1,2$. Theorem \ref{metric_d} will serve as the main tool in our study. 

\medskip
\noindent{\bf Stochastic heat equations with fractional noises.} Stochastic heat equations have been studied in various settings: with a time-space white noise, with generalizations of noise in space and/or in time, and also by considering differential operators more general than the Laplacian. The case of fractional noises was considered, e.g., in the recent papers \cite{BJQ-S}, \cite{HLS}, among many others (see references therein).  In the present paper we consider the stochastic heat equation:
\begin{align}\label{intrdu/dt} 
&\frac{\partial u}{\partial t} = \frac{\partial^2 u}{\partial x^2} + \frac{\partial^2 W}{\partial t \partial x}, ~~ (t, x) \in (0, T]\times\mathbb{R},
\\ \label{intru0}
&u(0, x) = u_0(x), x \in  \mathbb{R},
\end{align}
where
$W$ is a centered Gaussian process which is white in time and is fractional Brownian motion in space with index $H\le\frac{1}{2}$. The problem \eqref{intrdu/dt}--\eqref{intru0} was considered, for example, in   \cite{HLS} and it was stated that under some  continuity and boundedness conditions  on $u_0(x)$ there exists a unique mild solution $u(t, x)$,  $(t, x) \in (0, T]\times\mathbb{R}$, satisfying   the H\"{o}lder condition  
\begin{equation}\label{H}\|u(t, x) - u(s, y)\|_{L^p} \le c\Delta((t, x), (s, y))^{\rho\wedge H}\end{equation}
with some constant $c = c(p, T, H)$ and where $\Delta((t, x), (s, y)) = |t - s|^{\frac{1}{2}} + |x - y|$ is a parabolic metric, $\rho$ is an index in the H\"{o}lder condition imposed on $u_0(x)$. For our consideration, the bounds for the increments $u(t, x) - u(s, y)$ in $L_2$ norm will be important. In view of this, we restate the bound \eqref{H} in  another form for the case of $p=2$, with the constant $c$ given by a closed expression, and then use this bound to derive the results on the distribution of suprema and on the rate of growth for random fields representing the solution.

\medskip
\noindent{\bf Contents.}
The paper is organized as follows. 
In Section \ref{Section2} we study $\varphi$-sub-Gaussian random fields with $\varphi=\frac{|x|^\alpha}{\alpha}$, $\alpha\in(1,2]$. 
In Section \ref{Estimates} we present the estimates for the tail distribution of suprema on the bounded domain and in  
Section \ref{Rate} we state the results on the rate of growth of random fields over unbounded domains. 
Section \ref{Examples} presents   applications of the results of sections \ref{Estimates} and \ref{Rate} to random fields related to stochastic heat equations with fractional noise.

\section
{$Sub_\varphi (\Omega)$ processes with $\varphi=\frac{|x|^\alpha}{\alpha}$, $\alpha\in(1,2]$, defined on spaces with anisotropic metrics}\label{Section2}

Consider the process $X(t)$, $t\in T$, from the class of $\varphi$-sub-Gaussian processes with $\varphi=\frac{|x|^\alpha}{\alpha}$, $1 < \alpha\le2$. This class  is a natural generalisation of Gaussian and sub-Gaussian processes, which correspond to $\alpha = 2$.

For the function $\varphi(x)=\frac{|x|^\alpha}{\alpha}$, $1 < \alpha\le2$, we have $\varphi^{(-1)}(x) = (\alpha x)^{1/\alpha}$, $x>0$, and the Young-Fenchel transform $\varphi^{*}(x) = \frac{|x|^\beta}{\beta}$, where $\beta > 0$ is such that $\frac{1}{\beta} + \frac{1}{\alpha} = 1$, that is, $\beta = \frac{\alpha}{\alpha - 1}$. 

The entropy integrals \eqref{condI} and \eqref{condI2} become of the form

$$I_\varphi(\varepsilon) =\int\limits_0^{\varepsilon} \Big(\ln (N_{\rho_{X}} (u))) \Big)^{\frac{1}{\beta}}\,d u $$
and
\begin{equation}\label{condI3}
 \widetilde I_\varphi(\varepsilon) = \int\limits_0^{\varepsilon} \Big(\ln (N_{d} (\sigma^{(-1)} (u))) \Big)^{\frac{1}{\beta}}\,d u,
\end{equation}
and the bounds in Theorems \ref{sup_X} and \ref{metric_d} will be based on these integrals.

As we can see, for such function $\varphi$ the integrals appear in a quite simple form and can be evaluated for particular metrics $d$. Note that for more general $\varphi$ sometimes it is more convenient to use entropy integrals of another form (see, e.g., \cite{HS, KO2016, KOSV}).

We next consider $\varphi$-sub-Gaussian fields $X(t), t \in T$, with the parameter set $T = [a_1, b_1]\times[a_2, b_2]$ endowed with the  anisotropic metric 

\begin{equation}\label{metric}
d(t, s) = \sum_{i = 1, 2} |t_i - s_i|^{H_i}, H_i \in (0, 1], i = 1, 2.
\end{equation}
\

\subsection{Estimates for the distribution of suprema}\label{Estimates}

\begin{theorem}\label{3}
Let $X(t), t \in T$, be a $\varphi$-sub-Gaussian field with $\varphi (x)=\frac{|x|^\alpha}{\alpha}$, $\alpha\in(1,2]$, $T = [a_1, b_1]\times[a_2, b_2]$ with the metric $d(t, s)$ defined by \eqref{metric}, $\beta = \frac{\alpha}{\alpha - 1}$. 
Suppose that the field $X$ satisfies conditions (i)-(iii) of Theorem \ref{metric_d} with  $\sigma(h) = ch^\gamma$, $c > 0$, $\gamma \in (0, 1]$ and  $\gamma\beta\ne 1$.

Then for all $\lambda > 0$ and $\theta \in (0, 1)$
\begin{equation}\label{sp2.3} \E\exp\left\{\lambda\sup\limits_{t\in T}|X(t)|\right\}\le 2\exp\left\{\frac{1}{\alpha}\Big(\frac{\lambda \varepsilon_0}{1 - \theta} \Big)^\alpha + \frac{2 \lambda}{\theta(1 - \theta)} \Big(\theta\varepsilon_0 \Big)^{1 - \frac{1}{\gamma \beta}} c_1 \right\}
\end{equation}
and for all $\theta \in (0, 1), \theta\varepsilon_0 < \gamma_0$ and $u > \frac{2}{\theta(1 - \theta)} \big(\theta\varepsilon_0 \big)^{1 - \frac{1}{\gamma \beta}} c_1$,  we have 
\begin{equation}\prob \{\sup\limits_{t\in T} |X(t)| > u\} \leq 2\exp \Big\{- \frac{1}{\beta} \Big(\frac{u (1 - \theta)}{\varepsilon_0} - \frac{2}{\theta} \big(\theta\varepsilon_0 \big)^{1 - \frac{1}{\gamma \beta}} c_1 \Big)^\beta \Big\},
\end{equation}
where $c_1 = \frac{2^{\frac{1}{\beta}} c^{\frac{1}{\gamma\beta}}}{1 - \frac{1}{\gamma \beta}} 
\sum\limits_{i=1,2}\frac{1}{H_i}\big( \frac{T_i}{2} \big)^{\frac{H_i}{\beta}}$, with $T_i=b_i-a_i$, $i=1,2$. 
\end{theorem}

\begin{proof}
We apply Theorem \ref{metric_d}. We need to estimate the entropy integral $ \widetilde I_\varphi(\varepsilon)$ given by \eqref{condI2} for the particular $\varphi, \sigma, d$ under consideration. For the metric $d$ given by \eqref{metric} we can write the following estimate for the metric massivenes

$$N_d(\varepsilon) \le \prod_{i = 1, 2} \Big( \frac{T_i}{2(\frac{\varepsilon}{2})^{\frac{1}{H_i}}} + 1 \Big) = \prod_{i = 1, 2} \Big( \frac{2^{\frac{1}{H_i}} T_i}{2 \varepsilon^{\frac{1}{H_i}}} + 1 \Big).$$

This estimate can be deduced from the observation that a rectangle $[-(\frac{\varepsilon}{2})^{\frac{1}{H_1}}, (\frac{\varepsilon}{2})^{\frac{1}{H_1}}]\times[-(\frac{\varepsilon}{2})^{\frac{1}{H_2}}, (\frac{\varepsilon}{2})^{\frac{1}{H_2}}]$ is contained in the ball in metric $d$ with center $(0, 0)$ and radius $\varepsilon$, which is given as $B(\varepsilon) = \{(x_1, x_2) : |x_1|^{H_1} + |x_2|^{H_2} \le \varepsilon\}$.

Now, for the given functions $\varphi$ and $\sigma$ (note that $\sigma^{(-1)} (u) = (\frac{u}{c})^{\frac{1}{\gamma}}$), we consider the entropy integral \eqref{condI3}:

\begin{align*} \widetilde I_\varphi(\varepsilon) = \int\limits_0^{\varepsilon} \Big(\ln (N_{d} (\sigma^{(-1)} (u))) \Big)^{\frac{1}{\beta}}\,d u   &\le \int\limits_0^{\varepsilon} \Big(\ln \prod_{i = 1, 2} \Big( \frac{2^{\frac{1}{H_i}} T_i}{2 (\sigma^{(-1)} (u))^{\frac{1}{H_i}}} + 1 \Big) \Big)^{\frac{1}{\beta}}\,d u \\ & = \int\limits_0^{\varepsilon} \Big(\ln \prod_{i = 1, 2} \Big( \frac{T_i 2^{\frac{1}{H_i}} c^{\frac{1}{\gamma H_i}}}{2 u^{\frac{1}{\gamma H_i}}} + 1 \Big) \Big)^{\frac{1}{\beta}}\,d u.
\end{align*}
For any $0< \varkappa \le 1$, $\ln(1+x) = \frac{1}{\varkappa} \ln (1 + x)^\varkappa \le \frac{x^\varkappa}{\varkappa}$, we apply this inequality for each term in the product in the above formula choosing $\varkappa = H_i, i = 1, 2$: 

\begin{align*} \widetilde I_\varphi(\varepsilon) &\le \int\limits_0^{\varepsilon} \sum_{i = 1, 2} \frac{1}{H_i}\Big( \frac{T_i 2^{\frac{1}{H_i}} c^{\frac{1}{\gamma H_i}}}{2 u^{\frac{1}{\gamma H_i}}} \Big)^{\frac{H_i}{\beta}}\,d u =  \sum_{i = 1, 2}  \int\limits_0^{\varepsilon} \frac{1}{H_i}\Big( \frac{T_i}{2}\Big)^{\frac{H_i}{\beta}} 2^{\frac{1}{\beta}} c^{\frac{1}{\gamma\beta}}\, \frac{d u }{u^{\frac{1}{\gamma\beta}}}  \\ &= \frac{\varepsilon^{1 - \frac{1}{\gamma \beta}}}{1 - \frac{1}{\gamma \beta}}  2^{\frac{1}{\beta}} c^{\frac{1}{\gamma\beta}} \sum_{i = 1, 2} \frac{1}{H_i} \Big( \frac{T_i}{2}\Big)^{\frac{H_i}{\beta}}  = \varepsilon ^{1 - \frac{1}{\gamma \beta}} c_1 .
\end{align*}

\end{proof}

\subsection{Estimates for the rate of growth}\label{Rate}

Consider now the field $X(t_1, t_2)$, $(t_1, t_2) \in V$, defined over the unbounded domain $V=[0,+\infty)\times[-A, A]$.

Let $f(t)>0$, $t\ge 0$, be a continuous strictly increasing function and $f(t) \to \infty$ as $t \to \infty$.
 
Introduce the sequence $b_0 = 0$, $b_{k+1} > b_k$, $b_k \to \infty$, $k \to \infty$.  

We will use the following notations:

$l_k=b_{k+1}-b_k$, $V_k = [b_k, b_{k+1}]\times[-A, A]$, $k = 0,1, \dots$, 
$f_k = f(b_k)$, 

$\varepsilon_k = \sup_{(t_1, t_2) \in V_k} \tau_\varphi (X(t_1, t_2))$, and suppose that $0 < \varepsilon_k < \infty$;

$\gamma_k=c_k\max_{(t_1, t_2), (s_1, s_2)\in V_k} (d((t_1, t_2), (s_1, s_2)))^\gamma$, where $c_k$ are from \eqref{suptau} below, $\tilde\theta=\inf_k\frac{\gamma_k}{\varepsilon_k}$,

$\beta=\frac{\alpha}{\alpha-1}$.

\begin{theorem}\label{domain}
Let $X(t_1, t_2)$, $(t_1, t_2) \in V$, be a $\varphi$-sub-Gaussian separable field with $\varphi=\frac{|x|^\alpha}{\alpha}$, $\alpha\in(1,2]$. Supose further that 
\begin{enumerate}
\item[(i)] \begin{equation}\label{suptau}
\sup_{\substack{
        d((t_1, t_2), (s_1, s_2))\leq h,\\(t_1, t_2), (s_1, s_2)\in V_k
        }}\tau_{\varphi} (X (t_1, t_2) -
X (s_1, s_2))\leq c_k h^\gamma,  c_k > 0, 0 < \gamma \le 1;
\end{equation}

\item[(ii)] $C =  \sum_{k = 0}^\infty \frac{\varepsilon_k}{f_k}< \infty$.

\item[(iii)] 
$S = \sum_{k = 0}^\infty \frac{\varepsilon^{1 - \frac{1}{\gamma \beta}}_k c_1(k)}{f_k}< \infty$, where $c_1(k) = \Big(\frac{1}{H_1}(l_k/2)^{\frac{H_1}{\beta}} +\frac{1}{H_2} A^{\frac{H_2}{\beta}} \Big) \frac{ 2^{\frac{1}{\beta}} c_k^{\frac{1}{\gamma\beta}}}{1 - \frac{1}{\gamma \beta}}$.
\end{enumerate}

Then 
\begin{enumerate}
\item[(i)] for any $\theta \in (0, \min(1, \widetilde\theta))$ and any $\lambda > 0$

\begin{equation}
\E\exp\left\{\lambda\sup\limits_{(t_1, t_2) \in V}\frac{|X(t_1, t_2)|}{f(t_1)}\right\}\le 2\exp\left\{\frac{\lambda^\alpha}{\alpha}\left(\frac{C}{1-\theta}\right)^\alpha + \frac{2\lambda}{(1-\theta) \theta^{\frac{1}{\gamma \beta}}}S\right\};
\end{equation}

\item[(ii)] for any $\theta \in (0, \min(1, \widetilde\theta))$ and $u > \frac{2S}{(1-\theta) \theta^{\frac{1}{\gamma \beta}}}$

\begin{equation}\label{Prob}
\prob\left\{\sup\limits_{(t_1, t_2) \in V}\frac{|X(t_1, t_2)|}{f(t_1)} > u \right\}\le 2\exp\left\{-\frac{1}{\beta C^\beta}\left(u(1 - \theta) - 2S \theta^{-\frac{1}{\gamma \beta}}\right)^{\beta} \right\}.
\end{equation}
\end{enumerate}
\end{theorem}

\begin{proof} We use the scheme of the proof which is similar that in \cite{Dozzi} (Theorem 2.4), \cite{HS} (Theorem 4). Using \eqref{sp2.3} we can write the estimate:
\begin{align*} 
I(\lambda) &= \E\exp\left\{\lambda\sup\limits_{(t_1, t_2) \in V}\frac{|X(t_1, t_2)|}{f(t_1)}\right\}\le  \E\exp\left\{\lambda \sum_{k = 0}^\infty \sup\limits_{(t_1, t_2) \in V_k}\frac{|X(t_1, t_2)|}{f_k}\right\}\le \\ & \le \prod_{k = 0}^\infty \Big(\E\exp\left\{\lambda \frac{r_k}{f_k} \sup\limits_{(t_1, t_2) \in V_k}|X(t_1, t_2)|\right\}\Big)^{\frac{1}{r_k}}\le 2  \prod_{k = 0}^\infty Q_k(\lambda, \theta)^{\frac{1}{r_k}},
\end{align*}
where 
$$Q_k(\lambda, \theta) = \exp\left\{\frac{1}{\alpha}\Big(\frac{\lambda \varepsilon_k r_k}{(1 - \theta)f_k} \Big)^\alpha + 2\lambda \frac{r_k}{f_k} \frac{1}{(1 - \theta)\theta^{\frac{1}{\gamma\beta}}} \varepsilon_k^{1 - \frac{1}{\gamma \beta}} c_1(k) \right\},$$
$$c_1(k) = \Big(\frac{1}{H_1}(l_k/2)^{\frac{H_1}{\beta}} +\frac{1}{H_2} A^{\frac{H_2}{\beta}} \Big) \frac{ 2^{\frac{1}{\beta}} c_k^{\frac{1}{\gamma\beta}}}{1 - \frac{1}{\gamma \beta}},$$
and let here $r_k$, $k\ge 0$ are such that $\sum_{k=0}^{\infty}\frac{1}{r_k}=1$.  This implies that 
$$I(\lambda)  \le 2  \exp\left\{\frac{1}{\alpha}\Big(\frac{\lambda}{1 - \theta} \Big)^\alpha \sum_{k = 0}^\infty \Big(\frac{\varepsilon_k r_k}{f_k} \Big)^\alpha r_k^{-1} + \frac{2\lambda}{(1 - \theta)\theta^{\frac{1}{\gamma\beta}}}\sum_{k = 0}^\infty \frac{ \varepsilon_k^{1 - \frac{1}{\gamma \beta}} c_1(k)}{f_k} \right\}.$$

Choose $r_k = \frac{f_k}{ \varepsilon_k} C$, where $C =  \sum_{k = 0}^\infty \frac{\varepsilon_k }{f_k} $, then
$$I(\lambda)  \le  2  \exp\left\{\frac{1}{\alpha}\Big(\frac{\lambda}{1 - \theta} \Big)^\alpha C^\alpha +  \frac{2\lambda}{(1 - \theta)\theta^{\frac{1}{\gamma\beta}}} \sum_{k = 0}^\infty \frac{ \varepsilon_k^{1 - \frac{1}{\gamma \beta}} c_1(k)}{f_k} \right\}. $$

Therefore,
$$\E\exp\left\{\lambda\sup\limits_{(t_1, t_2) \in V}\frac{|X(t_1, t_2)|}{f(t_1)}\right\}\le 2  \exp\left\{\frac{\lambda^\alpha}{\alpha}\Big(\frac{C}{1 - \theta} \Big)^\alpha +  \frac{2\lambda}{(1 - \theta)\theta^{\frac{1}{\gamma\beta}}} S \right\}, $$
and for all $\lambda > 0$, $u > 0$, $0 < \theta < \widetilde\theta$
\begin{align*} 
P\left\{\sup\limits_{(t_1, t_2) \in V}\frac{|X(t_1, t_2)|}{f(t_1)} > u \right\}&\le \exp\{-\lambda u\}\E\exp\left\{\lambda\sup\limits_{(t_1, t_2) \in V}\frac{|X(t_1, t_2)|}{f(t_1)}\right\}\le \\ & \le 2  \exp\left\{\frac{\lambda^\alpha}{\alpha}\Big(\frac{C}{1 - \theta} \Big)^\alpha +  \frac{2\lambda S}{(1 - \theta)\theta^{\frac{1}{\gamma\beta}}} - \lambda u\right\}.
\end{align*}
We minimise the right-hand side with respect to $\lambda$ and obtain for $u > \frac{2S}{(1 - \theta)\theta^{\frac{1}{\gamma\beta}}}$
\begin{align*} 
P\left\{\sup\limits_{(t_1, t_2) \in V}\frac{|X(t_1, t_2)|}{f(t_1)} > u \right\}&\le 
2  \exp\left\{- \Big(u - \frac{2S}{(1 - \theta)\theta^{\frac{1}{\gamma\beta}}} \Big)^{\frac{\alpha}{\alpha - 1}} \Big(\frac{C}{1 - \theta} \Big)^{- \frac{\alpha}{\alpha - 1}} \frac{\alpha - 1}{\alpha} \right\} = \\ & = 2  \exp\left\{-  \frac{\alpha - 1}{\alpha} C^{- \frac{\alpha}{\alpha - 1}}\left(u(1 - \theta) - 2S \theta^{-\frac{1}{\gamma \beta}}\right)^{\frac{\alpha}{\alpha - 1}} \right\}.
\end{align*}
\end{proof}

\begin{theorem}\label{Th2.3}
Let $X(t_1, t_2)$, $(t_1, t_2) \in V$, be a $\varphi$-sub-Gaussian separable field with $\varphi=\frac{|x|^\alpha}{\alpha}$, $\alpha\in(1,2]$. Supose further that 
\begin{enumerate}
\item[(i)] \begin{equation*}
\sup_{\substack{
        d((t_1, t_2), (s_1, s_2))\leq h,\\(t_1, t_2), (s_1, s_2)\in V_k
        }}\tau_{\varphi} (X (t_1, t_2) -
X (s_1, s_2))\leq c_k h^\gamma,  c_k > 0, 0 < \gamma \le 1;
\end{equation*}

\item[(ii)] there exists $\delta > 0$ and a constant $c(\delta)$ such that for $(t_1, t_2) \in V$ 
$$\tau_{\varphi} (X (t_1, t_2)) \le c(\delta) t_1^\delta;$$

\item[(iii)] 
$\widetilde C =  c(\delta) \sum_{k = 0}^\infty \frac{b_{k+1}^\delta}{f_k}< \infty$,

\item[(iv)] 
$\widetilde S = (c(\delta))^{1 - \frac{1}{\gamma \beta}} \sum_{k = 0}^\infty \frac{b_{k+1}^{1 - \frac{1}{\gamma \beta}} c_1(k)}{f_k}< \infty$, where $c_1(k) = \Big(\frac{1}{H_1}(l_k/2)^{\frac{H_1}{\beta}} +\frac{1}{H_2} A^{\frac{H_2}{\beta}} \Big) \frac{ 2^{\frac{1}{\beta}} c_k^{\frac{1}{\gamma\beta}}}{1 - \frac{1}{\gamma \beta}}$.
\end{enumerate}

Then 
\begin{enumerate}
\item[(i)] for any $\theta \in (0, \min(1, \widetilde\theta))$ and any $\lambda > 0$ 
\begin{equation}
\E\exp\left\{\lambda\sup\limits_{(t_1, t_2) \in V}\frac{|X(t_1, t_2)|}{f(t_1)}\right\}\le 2\exp\left\{\frac{\lambda^\alpha}{\alpha}\left(\frac{\widetilde C}{1-\theta}\right)^\alpha + \frac{2\lambda}{(1-\theta) \theta}\widetilde S \right\};
\end{equation}

\item[(ii)] for any $\theta \in (0, \min(1, \widetilde\theta))$ and $u > \frac{2 (c(\delta))^{1 - \frac{1}{\gamma \beta}} \widetilde S}{(1-\theta) \theta^\frac{1}{\gamma \beta}}$
\begin{equation}
\prob\left\{\sup\limits_{(t_1, t_2) \in V}\frac{|X(t_1, t_2)|}{f(t_1)} > u \right\}\le 2\exp\left\{-\frac{1}{\beta \widetilde C^\beta}\left(u(1 - \theta) - 2 \widetilde S \theta^{\frac{1}{\gamma \beta}}\right)^{\beta} \right\}.
\end{equation}
\end{enumerate}
\end{theorem}
The proof is analogous to that of Theorem \ref{domain}.

\begin{remark}\label{Rem2.1} The bound \eqref{Prob} can be rewritten in another form. We can choose $\theta = u^{-\frac{\gamma \beta}{\gamma \beta + 1}}$ and then under conditions of Theorem \ref{3} for $u > (1 + 2S)^\frac{\gamma \beta}{\gamma \beta + 1}$ the following bound holds
$$\prob\left\{\sup\limits_{(t_1, t_2) \in V}\frac{|X(t_1, t_2)|}{f(t_1)} > u \right\}\le 2\exp\left\{-\frac{1}{\beta C^\beta}\left(u - u^\frac{1}{\gamma \beta +1}(1 +2S)\right)^{\beta} \right\}.$$

Correspondingly, under conditions of Theorem \ref{Th2.3} for $u > (1 + 2 \widetilde S)^\frac{\gamma \beta}{\gamma \beta + 1}$  we can write the bound
$$\prob\left\{\sup\limits_{(t_1, t_2) \in V}\frac{|X(t_1, t_2)|}{f(t_1)} > u \right\}\le 2\exp\left\{-\frac{1}{\beta \widetilde C^\beta}\left(u - u^\frac{1}{\gamma \beta +1}(1 +2\widetilde S)\right)^{\beta} \right\}.$$
\end{remark}

\begin{theorem}\label{Th2.4}
Suppose that for the field $X(t_1, t_2)$, $(t_1, t_2) \in V$, and function $f$ conditions of Theorem \ref{Th2.3} are satisfied. Then there exists a random variable $\xi$ such that with probability one
\begin{equation}\label{xxi}|X(t_1, t_2)| \le f(t_1)\xi,\end{equation}
and $\xi$ satisfies the following assumption:
\begin{equation}\label{**}
\prob\left\{\xi  > u \right\}\le 2\exp\left\{-\frac{1}{\beta \widetilde C^\beta}\left(u - u^\frac{1}{\gamma \beta +1}(1 +2\widetilde S)\right)^{\beta} \right\}
\end{equation}
for $u > (1 + 2 \widetilde S)^\frac{\gamma \beta}{\gamma \beta + 1}$.
\end{theorem}
\begin{proof} The theorem is a corollary of Theorem \ref{Th2.3}. We denote $\xi=\sup\limits_{(t_1, t_2) \in V}\frac{|X(t_1, t_2)|}{f(t_1)}$, then $\xi$ satisfies \eqref{**} and for all $\omega$ we have \eqref{xxi}.
\end{proof}

\section{Application to stochastic heat equation with fractional noise}\label{Examples}
\
 Consider the stochastic heat equation:
\begin{align}\label{du/dt} 
&\frac{\partial u}{\partial t} = \frac{\partial^2 u}{\partial x^2} + \frac{\partial^2 W}{\partial t \partial x}, ~~ (t, x) \in (0, T]\times\mathbb{R},
\\ \label{u0}
&u(0, x) = u_0(x), x \in  \mathbb{R},
\end{align}
with the following assumption about the noise:

$\bf{A.1.}$
$W$ is a centered Gaussian process which is white in time and is fractional Brownian motion in space with index $H\le\frac{1}{2}$.

Note that there exists an extensive literature on SPDEs driven by Gaussian noises and involving various differential operators.

As for the stochastic heat equation, it has been studied in various settings: as that one driven by time-space white noise, further with generalization of noise in space or in time, and also by considering differential operators more general than the Laplacian.

Since we are going to apply the results of the previous section, we are interested in the results for solution to the problem \eqref{du/dt}-\eqref{u0}, which allow to write down the bounds for $\E|u(t, x) - u(s, y)|^2$.

We will base our study on the results from the recent papers \cite{BJQ-S, HLS} on reqularity properties of solution to stochastic heat equation. 

We consider the problem \eqref{du/dt}-\eqref{u0} with assumption A.1. on the noise and imposing the following assumption on the initial condition $u_0$.

$\bf{A.2.}$ The process $u_0(x), x \in  \mathbb{R}$, is continuous, possesses uniformly bounded $p$-th moments for $p\ge 2$ and is stochastically H\"{o}lder continuous with $\rho \in (0, 1]$, i.e., for all $p \ge 1$

$$\|u_0(x) - u_0(y)\|_{L^p} \le L_0(p)|x - y|^\rho, ~~x, y \in \mathbb{R},$$
where $L_0(p)$ is a constant and $\|.\|_{L^p}$ denotes the norm in $L^p(\Omega,  \mathbb{R})$: $\|u\|_{L^p} =\left(\E|u|^p\right)^{\frac{1}{p}} $, $p\ge 1$.

From \cite{HLS} (see Theorem 1.1) it follows that under assumptions A.1. and A.2. equation \eqref{du/dt} has a unique mild solution $u(t, x),  (t, x) \in (0, T]\times\mathbb{R}$, satisfying  $\sup\limits_{(t, x) \in (0, T]\times\mathbb{R}} \E|u(t, x)|^p<\infty$ and the H\"{o}lder condition  
$$\|u(t, x) - u(s, y)\|_{L^p} \le c\Delta((t, x), (s, y))^{\rho\wedge H}$$
with some constant $c = c(p, T, H)$ and where $\Delta((t, x), (s, y)) = |t - s|^{\frac{1}{2}} + |x - y|$ is the parabolic metric.

The mild solution is defined as the random field
\begin{equation}\label{u(t,x)} 
u(t, x) = \int_\mathbb{R} G_t(x - y) u_0(y)\, d y + \int^t_0 \int_\mathbb{R} G_{t - \theta}(x - \eta) W(d \theta, d \eta) = \omega(t, x) + V(t, x),
\end{equation}
where $G_t(x)$ is the Green's function (fundamental solution) of the equation $(\frac{\partial}{\partial t} - \frac{\partial^2}{\partial x^2})u = 0$, that is,
$$G_t(x) = \frac{1}{\sqrt{4\pi t}} \exp \left(-\frac{|x|^2}{4t}\right).$$

We refer for the rigorous definitions of the integrals in \eqref{u(t,x)}, for example, to \cite{HLS}, \cite{BJQ-S}.

In the next theorem we state the result on H\"{o}lder continuity of the fields $\omega(t, x)$ and $V(t, x)$ in the form, where explicit expressions for the constants are given.

\begin{theorem}\label{bounds}
Let assumptions A.1. and A.2. hold. Then the following bounds hold:
\begin{equation}\label{3.4}
\|\omega(t, x)\|_{L^2} \le \sup\limits_{x \in \mathbb{R}}\|u_0(x)\|_{L^2} \le c_0;
\end{equation}
\begin{equation}\label{3.5}
\|V(t, x)\|_{L^2} \le A(H) t^{\frac{H}{2}};
\end{equation}
\begin{equation}\label{3.6} \|\omega(t, x) - \omega(s, y)\|_{L^2} \le c_\omega \left(|t - s|^{\frac{\rho}{2}} + |x - y|^\rho \right);\end{equation}
\begin{equation}\label{3.7} \|V(t, x) - V(s, y)\|_{L^2} \le c_V\left(|t - s|^{\frac{H}{2}} + |x - y|^H \right),\end{equation}
where the constants $c_\omega, c_V, A(H)$ are given by formulas \eqref{c_omega}, \eqref{c_V},  \eqref{A(H)} below.
\end{theorem}

\begin{proof} For the proof we use the reasoning similar to that applied for more general case in \cite{HLS} (Theorem 1.1)  (see also \cite{BJQ-S}). Therefore, we present only the main steps, our interest is in keeping the values of constants involved on each step through all the chain of bounds. We have
\begin{align*} 
\E|\omega(t, x) - \omega(s, x)|^2 &= \E \left|\int_\mathbb{R} \left(G_t(x - y) - G_s(x - y)\right)u_0(y)\, d y \right|^2  \\ & = \E \left|\int_\mathbb{R} G_{t - s} (y)\left(\int_\mathbb{R} G_s(x - z)(u_0(z - y) - u_0(z))\, dz \right)\, d y \right|^2  \\ & \le \int_\mathbb{R} G_{t - s} (y) \int_\mathbb{R} G_s(x - z)\E \left|u_0(z - y) - u_0(z) \right|^2 \, d z d y  \\ & \le L  \int_\mathbb{R} G_{t - s} (y) |y|^{2 \rho} \, d y = L C_1 |t - s|^\rho,
\end{align*}
since 
\begin{align*} 
\int_\mathbb{R} G_{h} (y) |y|^{2 \rho} \, d y &=  \int_\mathbb{R}  \frac{1}{\sqrt{4\pi h}} \exp \left(-\frac{y^2}{4h}\right) |y|^{2 \rho} \, dy   = \frac{1}{\sqrt{4\pi h}}\, 2 \int_0^\infty \exp \left(-\frac{y^2}{4h}\right) y^{2 \rho}\, dy \\ &=  \frac{4^\rho}{\sqrt{\pi}} \Gamma (\rho + \frac{1}{2}) h^\rho = C_1 h^\rho,
\end{align*}
where $C_1 = \frac{4^\rho}{\sqrt{\pi}} \Gamma (\rho + \frac{1}{2})$, $L=L_0(2)$, and we have used above the property $\int_\mathbb{R} G_{t} (y)\,dy=1$ and assumption A.2. We now consider the more general increment:
\begin{align*} 
\E|\omega(t, x) - \omega(s, y)|^2  & \le 2\left( \E|\omega(t, x) - \omega(s, x)|^2 + \E|\omega(s, x) - \omega(s, y)|^2\right)  \\ & \le 2 L C_1 |t - s|^\rho + 2 \E\left|\int_\mathbb{R} \left(G_s(x - z) - G_s(y - z)\right) u_0(z) \, d z  \right|^2  \\ & = 2 L C_1 |t - s|^\rho + 2 \E\left|\int_\mathbb{R} G_s(z) \left( u_0(x - z)  - u_0(y - z) \right) \, d z  \right|^2  \\ & \le 2 L C_1 |t - s|^\rho + 2 L^2 |x - y|^{2 \rho} \le c \left(|t - s|^\rho + |x - y|^{2 \rho} \right),
\end{align*}
where $c = 2 L \max(C_1, L)$, and using again assumption A.2. Therefore,
\begin{align} \label{c_omega}
\left( \E|\omega(t, x) - \omega(s, y)|^2 \right)^{\frac{1}{2}} \le  c_\omega \left(|t - s|^{\frac{\rho}{2}} + |x - y|^\rho \right),
~ c_\omega = \sqrt{c}.
\end{align}
Considering $V(t, x)$, we follow again the similar lines as those in \cite{HLS} keeping the track of constants. We can write:
\begin{align*} 
\E\left|V(t, x) - V(s, y)\right|^2 =& \E\left|\int_0^t \int_\mathbb{R} G_{t - \theta} (x - \eta) W(d \theta, d \eta) - \int_0^s \int_\mathbb{R} G_{s - \theta} (y - \eta) W(d \theta, d \eta)\right|^2 \\ =&
 \E\left|\int_0^t \int_\mathbb{R} G_{t - \theta} (x - \eta) W(d \theta, d \eta) \right.\\&+ \int_0^s \int_\mathbb{R} \left( G_{t - \theta} (x - \eta) - G_{s - \theta} (x - \eta)\right)W(d \theta, d \eta)   \\&+  \left. \int_0^s \int_\mathbb{R}  \left( G_{s - \theta} (x - \eta) - G_{s - \theta} (y - \eta) \right) W(d \theta, d \eta)\right|^2 \\ \le& 3C_H \left( \int_0^{t - s} \int_\mathbb{R} \left| F G_{\theta} (\xi) \right|^2 |\xi|^{1 - 2H} \, d\xi d\theta  \right. \\ &+ \int_0^{s} \int_\mathbb{R} \left| F G_{t - s + \theta} (\xi) -  F G_{\theta} (\xi) \right|^2 |\xi|^{1 - 2H} \, d\xi d\theta \\&+ \left.\int_0^{s} \int_\mathbb{R} \left(1 - \cos(\xi(x - y)) \right) |F G_{\theta} (\xi)|^2   |\xi|^{1 - 2H} \, d\xi d\theta \right)  \\=& 3 C_H (I_1 + I_2 + I_3),
\end{align*}
here above by $F$ attached to a function we mean the Fourier transform, we use that $F G_t (\xi) = \exp (- t \xi^2)$ and also the  isometry relation (see, e.g., \cite[Theorem 2.1]{HLS}, \cite[Theorem 2.7]{BJQ-S}):
$$
\E\left|\int_0^t \int_\mathbb{R} \varphi (s, y) W(ds, dy)\right|^2=\int_0^t \int_\mathbb{R} \left|F\varphi (s, \cdot)(y)\right|^2 \mu(dy)ds, 
$$
where $\mu$ is the spectral measure involved in the representation of the covariance of $W$ and $$\mu(dy)=C_H|y|^{1-2H}dy,\,\,\, C_H=\frac{\Gamma(2H+1)\sin(\pi H)}{2\pi}$$ (see for details \cite{HLS}, \cite{BJQ-S}). Evaluate now the integrals $I_1, I_2, I_3$.
$$I_1 =  \int_0^{t} \int_\mathbb{R} \exp (- 2 s \xi^2) |\xi|^{1 - 2H} \, d\xi d s = 2^{-(H + 1)} \frac{\Gamma (H + 1)}{H} t^H = c_{1,H} t^H,$$
where $c_{1,H} = 2^{-(H + 1)} \frac{\Gamma (H + 1)}{H}$;
\begin{align*} 
I_2 &= \int_0^{s} \int_\mathbb{R} \left| F G_{t - s + \theta} (\xi) -  F G_{\theta} (\xi) \right|^2 |\xi|^{1 - 2H} \, d\xi d\theta \\
&=  \int_0^{s} \int_\mathbb{R}  \exp (- 2 \theta \xi^2) \left( 1 -  \exp (- (t - s) \xi^2)\right)^2 \times  |\xi|^{1 - 2H} \, d\xi d\theta  \\
&= \int_\mathbb{R}  \frac{1 -  \exp (- 2 s \xi^2)}{2 |\xi|^2} \left( 1 -  \exp (- (t - s) \xi^2)\right)^2 |\xi|^{1 - 2H} \, d\xi  \\
&\le \frac{1}{2} \int_\mathbb{R}  \frac{\left( 1 -  \exp (- (t - s) \xi^2)\right)^2}{ |\xi|^{2 - 1 + 2H} } \, d\xi \\
&= (t - s)^H \frac{1}{2} \int_\mathbb{R}  \frac{\left( 1 -  \exp (- u^2)\right)^2}{ u^{1 + 2H} } \, du = c_{2,H} (t - s)^H,
\end{align*}
where $c_{2,H} = \frac{1}{2} \int_\mathbb{R}  \frac{\left( 1 -  \exp (- u^2)\right)^2}{ u^{1 + 2H} } \, du$;

\begin{align*} 
I_3 &=  \int_0^{s} \int_\mathbb{R}  \left(1 - \cos(\xi(x - y)) \right)  \exp (- 2 \theta \xi^2) |\xi|^{1 - 2H} \, d\xi d\theta  \\&= \int_\mathbb{R}  \frac{1 -  \exp (- 2 s \xi^2)}{2 |\xi|^2} \left(1 - \cos(\xi(x - y)) \right) |\xi|^{1 - 2H} \, d\xi   \\& \le \int_\mathbb{R}  \left(1 - \cos(\xi(x - y)) \right)  \frac{1}{2}  |\xi|^{- 1 - 2H} \, d\xi =   c_{3,H} (x - y)^{2H},
\end{align*}
where 

for $H <  \frac{1}{2}$: $c_{3,H} = \int_0^\infty \left(1 - \cos x \right) x^{- 1 - 2H} \, d x = (2 H)^{-1}  \Gamma (1 - 2H) \cos (H\pi)$,

for  $H = \frac{1}{2}$: $c_{3,H} =  \frac{\pi}{2}$. 

\noindent(We used here Lemma D.1. from \cite{BJQ-S}).

Therefore,
$$\E\left|V(t, x) - V(s, y)\right|^2 \le 3 C_H ((c_{1,H} +c_{2,H}) |t - s|^H +  c_{3,H} |x - y|^{2H})$$
and
$$\left( \E\left|V(t, x) - V(s, y)\right|^2 \right)^{\frac{1}{2}} \le  c_V \left(|t - s|^{\frac{H}{2}} + |x - y|^H \right),$$
where
\begin{align} \label{c_V}c_V = \sqrt{3 C_H \max ((c_{1,H} +c_{2,H}), c_{3,H})}.
\end{align}
We also have the bound
$$\left( \E\left|V(t, x)\right|^2 \right)^{\frac{1}{2}} \le A(H) t^{\frac{H}{2}},$$
where 
\begin{align} \label{A(H)} A(H) = \sqrt{C_H c_{1,H}}.
\end{align}
\end{proof}

For the next results we will use an additional assumption on the initial condition $u_0$ and will need the following definition(\cite{KK}): 

\noindent A family $\Delta$ of $\varphi$-sub-Gaussian random variables is called strictly $\varphi$-sub-Gaussian if
there exists a constant $C_\Delta $ (called determining constant) such that for all countable sets $I$ of random
 variables ${\zeta_i}\in \Delta$, $i\in I$, the  inequality holds:
$
\tau_\varphi \big( \sum_{i \in I} \lambda_i \zeta_i \big) \leq
 C_\Delta \big( \E \left( \sum_{i \in I}\lambda_i \zeta_i \right)^2\big)^{1/2}.$ Random field $\zeta(t)$, $t\in T$, is called
strictly $\varphi$-sub-Gaussian  if the family of random variables
$\{\zeta(t), t\in T\}$ is strictly $\varphi$-sub-Gaussian.

Let $K$ be a deterministic kernel and 
$\eta(x) =\int_T K(x,y) d\xi(y)$, $x\in X,
$
where $\xi(y)$, $y\in Y$,  is a strictly $\varphi$-sub-Gaussian  field
and the integral is defined in the mean-square sense. Then 
$\eta(x)$, $x\in X$, is  strictly $\varphi$-sub-Gaussian 
field with the same determining constant (see \cite{KK}).

\begin{theorem} Let assumptions A.1 and A.2 hold and assume that $u_0(x)$ is strictly $\varphi$-sub-Gaussian random field with $\varphi (x) =\frac{|x|^\alpha}{\alpha}$, $\alpha\in(1,2]$, and let  $c_\varphi$ be the determining constant. Then for fields $\omega(t, x)$ and $V(t, x)$,  $(t, x) \in D = [a_1, b_1]\times[a_2, b_2]$, the following estimates hold:
\begin{enumerate}
\item[(i)] 
\begin{equation}\label{supomega}\prob\left\{\sup\limits_{(t, x) \in D} |\omega(t, x)| > u \right\}\le 2\exp\left\{-\frac{1}{\beta}\left( \frac{u(1 - \theta)}{c_0c_\varphi} + 2 (\theta c_0c_\varphi)^{-\frac{1}{\beta}} \widetilde c_1 \right)^{\beta} \right\}\end{equation}
for $u > \frac{2}{(1-\theta) \theta} (\theta c_0c_\varphi)^{1 -\frac{1}{\beta}} \widetilde c_1 $, where  $\widetilde c_1 = \frac{2^{\frac{1}{\beta}} (c_\omega c_\varphi)^{\frac{1}{\beta}}}{1 - \frac{1}{\beta}}\left( \frac{2}{\rho}\big( \frac{T_1}{2} \big)^{\frac{\rho}{2\beta}} + \frac{1}{\rho}\big( \frac{T_2}{2} \big)^{\frac{\rho}{\beta}} \right)$, $T_i=b_i-a_i$, $i=1,2$, $\beta = \frac{\alpha}{\alpha - 1}$. 

\item[(ii)] 
$$\prob\left\{\sup\limits_{(t, x) \in D} |V(t, x)| > u \right\} \le 2\exp\left\{-\frac{1}{2}\left( \frac{u(1 - \theta)}{\varepsilon_V} + 2 (\theta \varepsilon_V)^{-\frac{1}{2}} \widetilde{\widetilde c}_1 \right)^{2} \right\}$$
for $u > \frac{2}{(1-\theta) \theta} (\theta \varepsilon_V)^{\frac{1}{2}} \widetilde{\widetilde c}_1$, where $\varepsilon_V = A(H)b_1^{H/2}$, $\widetilde{\widetilde c}_1 = 2\sqrt{2 c_V}\left( \frac{2}{H}\big( \frac{T_1}{2} \big)^{\frac{H}{4}} + \frac{1}{H}\big( \frac{T_2}{2} \big)^{\frac{H}{2}} \right)$, $T_i=b_i-a_i$, $i=1,2$.
\end{enumerate}
\end{theorem}

\begin{proof} Statement $(i)$ follows from Theorem \ref{3} and the estimates \eqref{3.4},\eqref{3.6} and we use that $\tau_\varphi(\omega(t,x))\le c_\varphi \|\omega(t,x)\|_{L^2}$. Statement $(ii)$ follows from Theorem \ref{3} with $\alpha = 2$ (since the field $V(t, x)$ is Gaussian) and the estimate \eqref{3.7}.
\end{proof}

The next theorem presents the power upper bound for the asymptotic growth of the trajectories of the field $V(t,x)$.

\begin{theorem}\label{Th3.3}  Let assumptions A.1 and A.2 hold. Then for any $p >1$ there exist a random variable $\xi(p)$ such that for any $(t, x) \in V$, with probability one
$$|V(t, x)| \le ((t^\frac{H}{2}(\log t)^p) \vee 1) \xi(p),$$
where $\xi(p)$ satisfies assumption \eqref{**} with $\gamma = 1, \beta = 2$ and some constants $\widetilde C$ and $\widetilde S$.
\end{theorem}

\begin{proof}
We apply Theorem \ref{Th2.4} and  Theorem \ref{Th2.3}. Condition $(i)$ holds since $V(t, x)$ is continuous, condition $(ii)$ holds in view of \eqref{3.5}.

Consider conditions $(iii)$ and $(iv)$. Let $f(t) = (t^\frac{H}{2}|\log t|^p) \vee 1$ for $t > 0$ and some $p >1$. Let us choose $b_k = e^k, k \ge 0$.

Then $f_k = b_k^\frac{H}{2}(\log b_k)^p = e^{k \frac{H}{2}}(\log e^k)^p = e^{k \frac{H}{2}} k^p$ and we have 
$$\widetilde C = A(H) \sum_{k = 0}^\infty \frac{b_{k+1}^\frac{H}{2}}{f_k} = A(H)\left( e^\frac{H}{2} + \sum_{k = 1}^\infty \frac{e^\frac{H}{2}}{k^p}\right) = A(H) e^\frac{H}{2} \left( 1 + \sum_{k = 1}^\infty \frac{1}{k^p}\right) < \infty.$$

Consider
$$\widetilde S = (A(H))^\frac{1}{2} \sum_{k = 0}^\infty \frac{b_{k+1}^\frac{H}{4} c_1(k)}{f_k} =  (A(H))^\frac{1}{2} \sum_{k = 0}^\infty \frac{e^{(k+1) \frac{H}{4}} c_1(k)}{e^{k \frac{H}{2}} k^p},$$
where $c_1(k) =  \frac{2}{H} e^{k \frac{H}{4}}  (\frac{e - 1}{2})^\frac{H}{4} + A^\frac{H}{2},$ and we can see that $\widetilde S < \infty$.

Therefore, conditions of Theorem \ref{Th2.4}  hold true, and applying this theorem with $f(t) = (t^\frac{H}{2}|\log t|^p) \vee 1$ we obtain the statement of Theorem \ref{Th3.3}.

\end{proof}

We consider now an assumption on the initial condition which can be used instead of assumption A.2.

$\bf{A.3.}$ The process 
 $u_0(x), x \in \mathbb{R}$  is a real, measurable, mean-square continuous stationary stochastic process.

Let $B(x), x \in \mathbb{R},$ is a covariance function of the process $u_0(x), x \in \mathbb{R}$, with the representation 
\begin{equation}\label{4.3}
B(x) = \int_\mathbb{R} \cos(\lambda x) dF(\lambda),
\end{equation}
where $F(\lambda)$ is a spectral measure, and for the process itself we can write the spectral representation
\begin{equation}\label{4.4} 
u_0(x) = \int_\mathbb{R} e^{i \lambda x} Z(d\lambda),
\end{equation}

The stochastic integral is considered as $L_2(\Omega)$ integral. Orthogonal random measure $Z$ is such that $\E|Z(d\lambda)|^2 = F(d\lambda)$. 

Then the field $\omega$ can be writen in the form
\begin{equation}\label{ex2}
\omega(t, x) = \int_\mathbb{R} \exp\Big\{ i\lambda x - \mu t \lambda^2 \Big\} Z(d\lambda)\end{equation}
and its covariance function has the following representation 
\begin{equation}\label{Cov}
Cov \Big( \omega(t, x), \omega(s, y)\Big) = \int_\mathbb{R} \exp\Big\{i \lambda (x - y) - \mu \lambda^2(t - s) \Big\}F(d\lambda)\end{equation}
(see, \cite{HS}).

\begin{theorem}\label{pr1} 
Let assumption A.3 holds. Then
\begin{equation}\label{F}
\|\omega(t, x)\|_{L^2} \le  \Big(\int_\mathbb{R}F(d\lambda)\Big)^{\frac{1}{2}},
\end{equation}
and if for some $\varepsilon\in(0,\frac{1}{2}]$
\begin{equation}\label{spectral}c^2(\varepsilon):=\int_\mathbb{R}\lambda^{2\varepsilon}F(d\lambda)<\infty,\end{equation}  
then the following estimate holds:
\begin{equation}\label{incr_ex2}
\| \omega(t, x) - \omega(s, y)\|_{L^2}\le  c(\varepsilon) \big( 4^{1-\varepsilon}|x-y|^{2\varepsilon} +  |t-s|^{\varepsilon}\big)^{1/2}.
\end{equation}
\end{theorem}

\begin{proof} We have
$$\E\big( \omega(t, x) - \omega(s, y)\big)^2 = \int_\mathbb{R} |b(\lambda)|^2 F(d\lambda),$$
where 
$$b(\lambda) = \exp\{i\lambda x\} \exp\{-\mu \lambda^2 t\} - \exp\{i\lambda y\} \exp\{-\mu \lambda^2 s\},$$
and we can estimate
\begin{align*}|b(\lambda)|^2 &\le \Big(1 - \exp\Big\{- \mu \lambda^2 |t - s| \Big\} \Big)^2 + 4\sin^2 \Big(\frac{1}{2}\lambda (x - y)\Big) \\
&\le \Big(\min(\lambda^2|t-s|, 1)\Big)^2+4\min\Big(\frac{1}{2}|\lambda| |x-y|, 1\Big)^2\\
&\le \Big(\lambda^2|t-s|\Big)^{2\varepsilon_1}
+4\Big(\frac{1}{2}|\lambda| |x-y|\Big)^{2\varepsilon_2}
\end{align*}
for any $\varepsilon_1, \varepsilon_2 \in (0,1]$. Let us choose $\varepsilon:=\varepsilon_1=\varepsilon_2/2$, $\varepsilon \in (0,1/2]$, and suppose $\int_\mathbb{R}\lambda^{2\varepsilon}F(d\lambda)<\infty$. Then we can write the bound:
\begin{align*}\int_\mathbb{R} |b(\lambda)|^2 F(d\lambda) \le 
\Big(\int_\mathbb{R}\lambda^{2\varepsilon}F(d\lambda)\Big) \big(4^{1-\varepsilon}|x-y|^{2\varepsilon}+|t-s|^{\varepsilon}\big),
\end{align*}
which implies \eqref{incr_ex2}.The estimate \eqref{F} follows from \eqref{Cov}.
\end{proof}

In view of Theorem \ref{pr1}, under assumption A.3 and assuming $u_0$ to be strictly $\varphi$-sub-Gaussian, we can write the estimate for the tail distribution of supremum of $\omega(t,x)$ which is analogous to \eqref{supomega}, where the constants $c_0$ and $c_\omega$ will come now from  \eqref{F}, \eqref{spectral}.

\medskip
In the example below we present the process which can be used as initial condition, for which \eqref{spectral} is satisfied.

\begin{example}
Let  the initial condition $u_0(x)$, $x\in \R$, be a Gaussian stationary process with the spectral density
\begin{equation}\label{7.1}
f(\lambda)=\frac{\sigma^2}{(1+\lambda^2)^{2\alpha}}, \ \lambda\in\R.
\end{equation}
The corresponding covariance function is of the form:
\begin{equation}\label{7.2}
B_\eta(x)=\frac{\sigma^2}{\sqrt\pi \Gamma(2\alpha)}\Big(\frac{|x|}{2}\Big)^{2\alpha-1/2}K_{2\alpha-1/2}(|x|), \ x\in\R,
\end{equation}
where $K_\nu$ is the modified Bessel function of the second kind, in particular,  $K_{1/2}(x)=\sqrt{\frac{\pi}{2x}}e^{-x}$. Covariances \eqref{7.2} constitute the so-called Mat\'ern class, a parameter $\nu=2\alpha-1/2>0$ controls the level of smoothness of the stochastic process. 

Note that the Gaussian stochastic process with the above covariance and spectral density can be obtained as solution to the following fractional partial differential equation:
\begin{equation*}\label{fracPDE}
\Big(1-\frac{d^2}{dx^2}\Big)^{\alpha}\eta(x)=w(x), \ x\in\R,
\end{equation*}
with $w$ being a white noise: $\E w(x)=0$ and $\E w(x)w(y)=\sigma^2\delta(x-y)$ (see, e.g. \cite[Thm. 3.1]{DOS}).

The Mat\'ern model is popular in spatial statistics and modeling random fields (with corresponding adjustment of \eqref{7.2} for $n$-dimensional case). The relation between the spatial Mat\'ern covariance model and stochastic partial differential equation 
$
(\mu-\Delta)^{\alpha}\eta(x)=w(x), \ x\in\R^n,
$
was established by Whittle in 1963 and since then has been widely used in various applied and theoretical contexts.

For the stationary initial condition $u_0$ with spectral density \eqref{7.1} the condition \eqref{spectral} holds and we are able to calculate the constant $c(\varepsilon)$ defined in \eqref{spectral}. 
We have 
$$
\int_\R \frac{\lambda^{2\varepsilon}}{(1+\lambda^2)^{2\alpha}}=\int_0^\infty \frac{t^{\varepsilon+1/2-1}}{(1+t)^{\varepsilon+1/2+2\alpha-\varepsilon-1/2}}\, dt = {\cal {B}}(\varepsilon+1/2, 2\alpha-\varepsilon-1/2),
$$
where ${\cal B}$ is Beta-function, $\varepsilon\in(0,1/2]$, $2\alpha-\varepsilon-1/2>0$ and we used the formula $\int_0^\infty \frac{t^{\mu-1}}{(1+t)^{\mu+\nu}}\, dt = {\cal {B}}(\mu, \nu)$.

Therefore, in this case we obtain $c^2(\varepsilon)=\sigma^{2}{\cal {B}}(\varepsilon+1/2, 2\alpha-\varepsilon-1/2)$. In particular, having in \eqref{7.1} $\alpha>1$ and choosing $\varepsilon=1/2$ we get $c(1/2)=\sigma^2 \frac{1}{\alpha-1}$.

\end{example}

{\small
\baselineskip=0.7\baselineskip

}	

\end{document}